\documentclass[12pt]{article}
\usepackage{e-jc}
\usepackage{amsmath}
\usepackage{amssymb}
\usepackage{amsthm}

\newtheorem{theorem}{Theorem}
\newtheorem{lemma}{Lemma}
\newtheorem{mainth}{Main Theorem}
\newtheorem{corollary}{Corollary}
\newtheorem{remark}{Remark}
\def\PG{\mathrm{PG}}

\def\PGammaL{\mathrm{P\Gamma L}}
\def\F{\mathbb{F}}
\def\V{\mathrm{V}}

\def\D{\mathcal{D}}
\def\S{\mathcal{S}}
\def\B{\mathcal{B}}
\title{On the linearity of higher-dimensional blocking sets}
\author{G. Van de Voorde\thanks{The author is supported by the Fund for Scientific Research Flanders (FWO -- Vlaanderen). }}
   \date{\dateline{June 16, 2010}{Nov 29, 2010}}

\begin{document}
\maketitle
\begin{abstract} 
A small minimal $k$-blocking set $B$ in $\PG(n,q)$, $q=p^t$, $p$ prime, is a set of less than $3(q^k+1)/2$ points in $\PG(n,q)$, such that every $(n-k)$-dimensional space contains at least one point of $B$ and such that no proper subset of $B$ satisfies this property. The {\em linearity conjecture} states that all small minimal $k$-blocking sets in $\PG(n,q)$ are linear over a subfield  $\mathbb{F}_{p^e}$ of $\mathbb{F}_q$. Apart from a few cases, this conjecture is still open. 
In this paper, we show that to prove the linearity conjecture for $k$-blocking sets in $\PG(n,p^t)$, with exponent $e$ and $p^e\geq 7$, it is sufficient to prove it for one value of $n$ that is at least $2k$. Furthermore, we show that the linearity of small minimal blocking sets in $\PG(2,q)$ implies the linearity of small minimal $k$-blocking sets in $\PG(n,p^t)$, with exponent $e$, with $p^e\geq t/e+11$. 
\end{abstract}

\section{Introduction and preliminaries}
If $\V$ is a vectorspace, then we denote the corresponding projective space by $\PG(V)$. If $\V$ has dimension $n$ over the finite field $\F_q$, with $q$ elements, $q=p^t$, $p$ prime, then we also write $\V$ as $\V(n,q)$ and $\PG(V)$ as $\PG(n-1,q)$. A $k$-dimensional space will be called a $k$-space.

A {\em $k$-blocking set} in $\PG(n,q)$ is a set $B$ of points such that every $(n-k)$-space of $\PG(n,q)$ contains at least one point of $B$. A $k$-blocking set $B$ is called {\em small} if $|B|<3(q^k+1)/2$ and {\em minimal} if no proper subset of $B$ is a $k$-blocking set. The points of a $k$-space of $\PG(n,q)$ form a $k$-blocking set, and every $k$-blocking set containing a $k$-space is called {\em trivial}. Every small minimal $k$-blocking set $B$ in $\PG(n,p^t)$, $p$ prime, has an {\em exponent e}, defined to be the largest integer for which every $(n-k)$-space intersects $B$ in $1$ mod $p^e$ points. The fact that every small minimal $k$-blocking set has an exponent $e\geq 1$ follows from a result of Sz\H{o}nyi and Weiner and will be explained in Section 2. A minimal $k$-blocking set $B$ in $\PG(n,q)$ is of {\em R\'edei-type} if there exists a hyperplane containing $\vert B\vert -q^k$ points of $B$; this is the maximum number possible if $B$ is small and spans $\PG(n,q)$. For a long time, all constructed small minimal $k$-blocking sets were of R\'edei-type, and it was conjectured that all small minimal $k$-blocking sets must be of R\'edei-type. In 1998, Polito and Polverino \cite{polito} used a construction of Lunardon \cite{Lunardon} to construct small minimal {\em linear} blocking sets that were not of R\'edei-type, disproving this conjecture.  Soon people conjectured that all small minimal $k$-blocking sets in $\PG(n,q)$ must be linear. In 2008, the `Linearity conjecture' was for the first time formally stated in the literature, by Sziklai \cite{sziklai}.

A point set $S$ in $\PG(V)$, where $\V$ is an $(n+1)$-dimensional vector space over ${\mathbb{F}_{p^t}}$, is called {\em linear} if there exists a subset $U$ of $\V$ that forms an $\F_{p_0}$-vector space for some $\F_{p_0} \subset {\mathbb{F}}_{p^t}$, such that $S=\B(U)$, where 
$$\B(U):=\{\langle u \rangle_{\mathbb{F}_{p^t}}~:~u \in U\setminus \{0\}\}.$$ 
If we want to specify the subfield we call $S$ an {\it $\F_{p_0}$-linear set} (of $\PG(n,p^t)$).

We have a one-to-one correspondence between the points of $\PG(n,p_0^h)$ and the elements of a Desarguesian $(h-1)$-spread $\D$ of $\PG(h(n+1)-1,p_0)$. This gives us a different view on linear sets; namely, an $\F_{p_0}$-linear set is a set $S$ of points of $\PG(n,p_0^h)$ for which there exists a subspace $\pi$ in $\PG(h(n+1)-1,p_0)$ such that the points of $S$ correspond to the elements of $\D$ that have a non-empty intersection with $\pi$. We identify the elements of $\D$ with the points of $\PG(n,p_0^h)$, so we can view $\B(\pi)$ as a subset of $\D$, i.e.
$$\B(\pi)=\{S\in \D|S\cap \pi\neq \emptyset\}.$$

If we want to denote the element of $\D$ corresponding to the point $P$ of $\PG(n,p_0^h)$, we write $\S(P)$, analogously, we denote the set of elements of $\D$ corresponding to a subspace $H$ of $\PG(n,p_0^h)$, by $\S(H)$. For more information on this approach to linear sets, we refer to \cite{linearsets}. 

To avoid confusion, subspaces of $\PG(n,p_0^h)$ will be denoted by capital letters, while subspaces of $\PG(h(n+1)-1,p_0)$ will be denoted by lower-case letters.


\begin{remark}\label{trans} The following well-known property will be used throughout this paper: if $\B(\pi)$ is an $\F_{p_0}$-linear set in $\PG(n,p_0^h)$, where $\pi$ is a $d$-dimensional subspace of $\PG(h(n+1)-1,p_0)$, then for every point $x$ in $\PG(h(n+1)-1,p_0)$, contained in an element of $\B(\pi)$, there is a $d$-dimensional space $\pi'$, through $x$, such that $\B(\pi)=\B(\pi')$. This is a direct consequence of the fact that the elementwise stabilisor of $\D$ in $\PGammaL(h(n+1),p_0)$ acts transitively on the points of one element of $\D$. \end{remark}

To our knowledge, the Linearity conjecture for $k$-blocking sets $B$ in $\PG(n,p^t)$, $p$ prime, is still open, except in the following cases:
\begin{itemize}
\item $t=1$ (for $n=2$, see \cite{blok}; for $n>2$, this is a corollary of Theorem \ref{szonyi} (i));
\item $t=2$ (for $n=2$, see \cite{TS:97}; for $k=1$, see \cite{Storme-Weiner}; for $k\geq 1$, see \cite{bokler} and \cite{weiner});
\item $t=3$ (for $n=2$, see \cite{pol}; for $k=1$, see \cite{Storme-Weiner}; for $k\geq 1$, see \cite{LC} and independently \cite{nora},\cite{nora1});
\item $B$ is of R\'edei-type (for $n=2$, see \cite{redei2}; for $n>2$, see \cite{redei});
\item $B$ spans an $tk$-dimenional space (see \cite[Theorem 3.14]{sz}).
\end{itemize}

It should be noted that in $\PG(2,p^t)$, for $t=1,2,3$, all small minimal blocking sets are of R\'edei-type. Storme and Weiner show in \cite{Storme-Weiner} that small minimal $1$-blocking sets in $\PG(n,p^t)$, $t=2,3$, are of R\'edei-type too. The proofs rely on the fact that for $t=2,3$, small minimal blocking sets in $\PG(2,p^t)$ are listed. The special case $k=1$ in Main Theorem 1 of this paper shows that using the (assumed) linearity of planar small minimal blocking sets, it is possible to prove the linearity of small minimal $1$-blocking sets in $\PG(n,p^t)$, which reproofs the mentioned statements of Storme and Weiner in the cases $t=2,3$.

The techniques developed in \cite{LC} to show the linearity of $k$-blocking sets in $\PG(n,p^3)$, using the linearity of $1$-blocking sets in $\PG(n,p^3)$, can be modified to apply for general $t$. This will be Main Theorem 2 of this paper. In particular, this theorem reproofs the results of \cite{weiner}, \cite{LC}, \cite{nora}, \cite{nora1}.

In this paper, we prove the following main theorems. Recall that the {exponent $e$} of a small minimal $k$-blocking set is the largest integer such that every $(n-k)$-space meets in $1$ mod $p^e$ points. Theorem \ref{szonyi} (i) will assure that the exponent of a small minimal blocking set is at least $1$. 

\begin{mainth} If for a certain pair $(k,n^\ast)$ with $n^\ast\geq 2k$, all small minimal $k$-blocking sets in $\PG(n^\ast,p^t)$ are linear, then for all $n>k$, all small minimal $k$-blocking sets with exponent $e$ in $\PG(n,p^t)$, $p$ prime, $p^e\geq 7$, are linear.
\end{mainth}
In particular, this shows that if the linearity conjecture holds in the plane, it holds for all small minimal $1$-blocking sets with exponent $e$ in $\PG(n,p^t)$, $p^e\geq 7$.

\begin{mainth} If all small minimal $1$-blocking sets in $\PG(n,p^t)$ are linear, then all small minimal $k$-blocking sets with exponent $e$ in $\PG(n,p^t)$, $n>k$, $p^e\geq t/e+11$, are linear.
\end{mainth}
Combining the two main theorems yields the following corollary.
\begin{corollary}
If the linearity conjecture holds in the plane, it holds for all small minimal $k$-blocking sets with exponent $e$ in $\PG(n,p^t)$, $n>k$, $p$ prime, $p^e\geq t/e+11$.
\end{corollary}

\section{Previous results}
In this section, we list a few results on the linearity of small minimal $k$-blocking sets and on the size of small $k$-blocking sets that will be used throughout this paper. The first of the following theorems of Sz\H{o}nyi and Weiner has the linearity of small minimal $k$-blocking sets in projective spaces over prime fields as a corollary.
\begin{theorem} \label{szonyi} Let $B$ be a $k$-blocking set in $\PG(n,q)$, $q=p^t$, $p$ prime.
\begin{itemize}
\item[(i)]{\rm \cite[Theorem 2.7]{sz}} If $B$ is small and minimal, then $B$ intersects every subspace of $\PG(n,q)$ in $1$ mod $p$ or zero points.
\item[(ii)]{\rm \cite[Lemma 3.1]{sz}} If $\vert B \vert\leq 2q^k$ and every $(n-k)$-space intersects $B$ in $1$ mod $p$ points, then $B$ is minimal.
\item[(iii)]{\rm \cite[Corollary 3.2]{sz}} If $B$ is small and minimal, then the projection of $B$ from a point $Q\notin B$ onto a hyperplane $H$ skew to $Q$ is a small minimal $k$-blocking set in $H$.
\item[(iv)]{\rm \cite[Corollary 3.7]{sz}} The size of a non-trivial $k$-blocking set in $\PG(n,p^t)$, $p$ prime, with exponent $e$, is at least $p^{tk}+1+p^e\lceil \frac{p^{tk}/p^e+1}{p^e+1}\rceil$.
\end{itemize}
\end{theorem}

Part (iv) of the previous theorem gives a lower bound on the size of a $k$-blocking set. In this paper, we will work with the following, weaker, lower bound.
\begin{corollary}  \label{grootte} The size of a non-trivial $k$-blocking set in $\PG(n,p^t)$, $p$ prime, with exponent $e$, is at least $p^{tk}+p^{tk-e}-p^{tk-2e}$.
\end{corollary}

If a blocking set $B$ in $\PG(2,q)$ is $\F_{p_0}$-linear, then every line intersects $B$ in an $\F_{p_0}$-linear set. If $B$ is small, many of these $\F_{p_0}$-linear sets are $\F_{p_0}$-sublines (i.e. $\F_{p_0}$-linear sets of rank 2). The following theorem of Sziklai shows that for {\em all} small minimal blocking sets, this property holds. 

\begin{theorem} \begin{itemize}
\item[(i)] {\rm\cite[Proposition 4.17 (2)]{sziklai}}\label{sziklai} If $B$ is a small minimal blocking set in $\PG(2,q)$, with $|B|=q+\kappa$, then the number of $(p_0+1)$-secants to $B$ through a point $P$ of $B$ lying on a $(p_0+1)$-secant to $B$, is at least $$q/p_0-3(\kappa-1)/p_0+2.$$
\item[(ii)] {\rm \cite[Theorem 4.16]{sziklai}} \label{linrechte} Let $B$ be a small minimal blocking set with exponent $e$ in $\PG(2,q)$. If for a certain line $L$, $\vert L\cap B \vert=p^e+1$, then $\F_{p^e}$ is a subfield of $\F_q$ and $L\cap B$ is $\F_{p^e}$-linear.
\end{itemize}
\end{theorem}

The next theorem, by Lavrauw and Van de Voorde, determines the intersection of an $\F_p$-subline with an $\F_p$-linear set; all possibilities for the size of the intersection that are obtained in this statement, can occur (see \cite{linearsets}). The bound on the characteristic of the field appearing in Main Theorem 2 arises from this theorem.
\begin{theorem}{\rm\cite[Theorem 8]{linearsets}} \label{subline}An $\mathbb{F}_{p_0}$-linear set of rank $k$ in $\PG(n,p^t)$ and an $\mathbb{F}_{p_0}$-subline (i.e. an $\mathbb{F}_{p_0}$-linear set of rank $2$), intersect in $0,1,2,\ldots,k$ or $p_0+1$ points.
\end{theorem}

The following lemma is a straightforward extension of \cite[Lemma 7]{LC}, where the authors proved it for $h=3$.
\begin{lemma}\label{lemma5}\label{groottes}
If $B$ is a subset of $\PG(n,p_0^h)$, $p_0\geq 7$, intersecting every $(n-k)$-space, $k\geq 1$, in $1\mod{p_0}$ points, and $\Pi$ is an $(n-k+s)$-space, $s< k$, then either
$$\vert B\cap \Pi \vert< p_0^{hs}+p_0^{hs-1}+p_0^{hs-2}+3p_0^{hs-3}$$
or
$$\vert B\cap \Pi \vert> p_0^{hs+1}-p_0^{hs-1}-p_0^{hs-2}-3p_0^{hs-3}.$$
Furthermore, $\vert B\vert <p_0^{hk}+p_0^{hk-1}+p_0^{hk-2}+3p_0^{hk-3}$.
\end{lemma}
\begin{proof}
Let $\Pi$ be an $(n-k+s)$-space of $\PG(n,p_0^h)$, $s\leq k$, and put $B_\Pi:=B\cap \Pi$.
Let $x_i$ denote the number of $(n-k)$-spaces of $\Pi$ intersecting $B_\Pi$ in $i$ points. 
Counting the number of $(n-k)$-spaces, the number of incident pairs $(P, \Sigma)$ with $
P\in B_\Pi, P\in \Sigma,\Sigma$ an $(n-k)$-space, and the number of triples $(P_1,P_2,\Sigma)
$, with $P_1,P_2\in B_\Pi$, $P_1\neq P_2$, $P_1,P_2\in \Sigma$, $\Sigma$ an $(n-k)$-space 
yields:
 \begin{eqnarray}
\sum_i x_i&=&\left[\begin{array}{c}n-k+s+1\\n-k+1 \end{array}\right]_{p_0^h},\\
\sum_i ix_i&=&\vert B_\Pi \vert \left[\begin{array}{c}n-k+s\\n-k \end{array}\right]_{p_0^h},\\
\sum i(i-1)x_i&=&\vert B_\Pi\vert (\vert B_\Pi\vert-1)\left[\begin{array}{c}n-k+s-1\\n-k-1 
\end{array}\right]_{p_0^h}.
\end{eqnarray}
Since we assume that every $(n-k)$-space intersects $B$ in $1\mod{p_0}$ points, it follows that 
every $(n-k)$-space of $\Pi$ intersect $B_\Pi$ in $1\mod{p_0}$ points, and hence
$\sum_i (i-1)(i-1-p_0)x_i\geq 0$.
Using Equations (1), (2), and (3), this yields that
$$\vert B_\Pi \vert (\vert B_\Pi\vert -1)(p_0^{hn-hk+h}-1)(p_0^{hn-hk}-1)-(p_0+1)\vert B_\Pi \vert (p_0^{hn-hk+hs}-1)(p_0^{hn-hk+h}-1)$$
$$+(p_0+1)(p_0^{hn-hk+hs+h}-1)(p_0^{hn-hk+hs}-1)\geq 0.$$ Putting $\vert B_\Pi\vert=p_0^{hs}+p_0^{hs-1}+p_0^{hs-2}+3p_0^{hs-3}$  in this inequality, with $p_0\geq7$, gives a contradiction; putting $\vert B_\Pi\vert=p_0^{hs+1}-p_0^{hs-1}-p_0^{hs-2}-3p_0^{hs-3}$ in this inequality, with $p_0\geq7$, gives a contradiction if $s<k$. For $s=k$, it is sufficient to note that when $\vert B\vert$ is the size of a $k$-space, the inequality holds, to deduce that $\vert B \vert <p_0^{hk}+p_0^{hk-1}+p_0^{hk-2}+3p_0^{hk-3}$. The statement follows.
\end{proof}
Let $B$ be a subset of $\PG(n,p_0^h)$, $p_0\geq 7$, intersecting every $(n-k)$-space, $k\geq 1$, in $1\mod{p_0}$ points. From now on, we call an $(n-k+s)$-space {\em small} if it meets $B$ in less than $p_0^{hs}+p_0^{hs-1}+p_0^{hs-2}+3p_0^{hs-3}$ points, and {\em large} if it meets $B$ in more than $p_0^{hs+1}-p_0^{hs-1}-p_0^{hs-2}-3p_0^{hs-3}$ points, and it follows from the previous lemma that each $(n-k+s)$-space is either small or large.

The following Lemma and its corollaries show that if all $(n-k)$-spaces meet a $k$-blocking set $B$ in $1$ mod $p_0$ points, then every subspace that intersects $B$, intersects it in $1$ mod $p_0$ points.

\begin{lemma}\label{e} Let $B$ be a small minimal $k$-blocking set in $\PG(n,p_0^h)$ and let $L$ be a line such that $1<\vert B\cap L \vert<p_0^h+1.$ For all $i\in \{1,\ldots,n-k\}$ there exists an $i$-space $\pi_i$ through $L$ such that $B\cap \pi_i=B\cap L$.
\end{lemma}
\begin{proof} 
It follows from Theorem \ref{szonyi} that every subspace through $L$ intersects $B\setminus L$ in zero or at least $p$ points, where $p_0=p^e$, $p$ prime. We proceed by induction on the dimension $i$. The statement obviously holds for $i=1$. Suppose there exists an $i$-space $\Pi_i$ through $L$ such that $\Pi_i\cap B$=$L\cap B$, with $i\leq n-k-1$. If there is no $(i+1)$-space intersecting $B$ only in points of $L$, then the number of points of $B$ is at least
$$
|B\cap L|+p(p_0^{h(n-i-1)}+p_0^{h(n-i-2)}+\ldots+p_0^h+1),
$$
but by Lemma \ref{groottes} $\vert B \vert \leq p_0^{hk}+p_0^{hk-1}+p_0^{hk-2}+p_0^{hk-3}$. If $i<n-k$ this is a contradiction. 
We may conclude that there exists an $i$-space $\Pi_i$ through $L$ such that $B\cap L=B\cap \Pi_i$, $\forall i\in \{1,\ldots,n-k\}$.
\end{proof}

Using Lemma \ref{e}, the following corollaries follow easily.
\begin{corollary}{\rm (see also \cite[Corollary 3.11]{sz})}\label{rechte1modp} Every line meets a small minimal $k$-blocking set in $\PG(n,p^t)$, $p$ prime, with exponent $e$ in $1$ mod $p^e$ or zero points.
\end{corollary}
\begin{proof} Suppose the line $L$ meets the small minimal $k$-blocking set in $x$ points, where $1\leq x\leq p^t$. By Lemma \ref{e}, the line $L$ is contained in an $(n-k)$-space $\pi$ such that $B\cap \pi=B\cap L$. Since every $(n-k)$-space meets the $k$-blocking set $B$ with exponent $e$ in $1$ mod $p^e$ points, the corollary follows.
\end{proof}
By considering all lines through a certain point of $B$ in some subspace, we get the following corollary.
\begin{corollary}{\rm (see also \cite[Corollary 3.11]{sz})} Every subspace meets a small minimal $k$-blocking set in $\PG(n,p^t)$, $p$ prime, with exponent $e$ in $1$ mod $p^e$ or zero points.
\end{corollary}

\section{On the $(p_0+1)$-secants to a small minimal $k$-blocking set}

In this section, we show that Theorem \ref{sziklai} on planar blocking sets can be extended to a similar result on $k$-blocking sets in $\PG(n,q)$. 

\begin{lemma}\label{lemma6} Let $B$ be a small minimal $k$-blocking set with exponent $e$ in $\PG(n,p_0^h)$, $p_0:=p^e\geq 7$, $p$ prime, $\mathbf{ n \geq 2k+1}$. The number of points, not in $B$, that do not lie on a secant line to $B$ is at least  $$(p_0^{h(n+1)}-1)/(p_0^h+1)-(p_0^{2hk-2}+2p_0^{2hk-3})(p_0^h+1)-p_0^{hk}-p_0^{hk-1}-p_0^{hk-2}-3p_0^{hk-3},$$ and this number is larger than the number of points in $\PG(n-1,p_0^h)$.
\end{lemma}
\begin{proof} By Corollary \ref{rechte1modp}, the number of secant lines to $B$ is at most $\frac{|B|(|B|-1)}{(p_0+1)p_0}$.  By Lemma \ref{lemma5}, the number of points in $B$ is at most $p_0^{hk}+p_0^{hk-1}+p_0^{hk-2}+3p_0^{hk-3}$, hence the number of secant lines is at most $p_0^{2hk-2}+2p_0^{2hk-3}$. This means that the number of points on at least one secant line is at most $(p_0^{2hk-2}+2p_0^{2hk-3})(p_0^h+1)$. It follows that the number of points in $\PG(n,p_0^h)$, not in $B$, not on a secant to $B$ is at least $(p_0^{h(n+1)}-1)/(p_0^h+1)-(p_0^{2hk-2}+2p_0^{2hk-3})(p_0^h+1)-p_0^{hk}-p_0^{hk-1}-p_0^{hk-2}-3p_0^{hk-3}$. Since we assume that $n\geq 2k+1$ and $p_0\geq 7$, the last part of the statement follows.
\end{proof}

 We first extend Theorem \ref{sziklai} (i) to $1$-blocking sets in $\PG(n,q)$.

\begin{lemma} \label{lemma1} A point of a small minimal $1$-blocking set $B$ with exponent $e$ in $\PG(n,p_0^h)$, $p_0:=p^e\geq 7$, $p$ prime, lying on a $(p_0+1)$-secant, lies on at least $p_0^{h-1}-4p_0^{h-2}+1$ $(p_0+1)$-secants.
\end{lemma}
\begin{proof} 
We proceed by induction on the dimension $n$. If $n=2$, by Theorem \ref{sziklai}, the number of $(p_0+1)$-secants through $P$ is at least $q/p_0-3(\kappa-1)/p_0+2$, where $|B|=q+\kappa$. By Lemma \ref{lemma5}, $\kappa$ is at most $p_0^{h-1}+p_0^{h-2}+3p_0^{h-3}$, which means that the number of $(p_0+1)$-secants is at least $p_0^{h-1}-4p_0^{h-2}+1$. This proves the statement for $n=2$.

Now assume $n\geq 3$. From Lemma \ref{lemma6} (observe that, since $n\geq 3$ and $k=1$, $n\geq 2k+1$), we know that there is a point $Q$, not lying on a secant line to $B$. Project $B$ from the point $Q$ onto a hyperplane through $P$ and not through $Q$. It is clear that the number of $(p_0+1)$-secants through $P$ to the projection of $B$ is the number of $(p_0+1)$-secants through $P$ to $B$. By the induction hypothesis, this number is at least $p_0^{h-1}-4p_0^{h-2}+1$.
\end{proof}

\begin{lemma}\label{large} Let $\Pi$ be an $(n-k)$-space of $\PG(n,p_0^h)$, $k>1$, $p_0\geq 7$. If $\Pi$ intersects a small minimal $k$-blocking set $B$ with exponent $e$ in $\PG(n,p_0^h)$, $p_0:=p^e\geq 7$, $p$ prime in $p_0+1$ points, then there are at most $3p_0^{hk-h-3}$ large $(n-k+1)$-spaces through $\Pi$.
\end{lemma}
\begin{proof}

Suppose there are $y$ large $(n-k+1)$-spaces through $\Pi$. A small $(n-k+1)$-space through $\Pi$ meets $B$ clearly in a small $1$-blocking set, which is in this case, non-trivial and hence, by Theorem \ref{grootte}, has at least $p_0^h+p_0^{h-1}-p_0^{h-2}$ points.

Then the number of points in $B$ is at least
$$ y(p_0^{h+1}-p_0^{h-1}-p_0^{h-2}-3p_0^{h-3}-p_0-1)+$$
$$((p_0^{hk}-1)/(p_0^h-1)-y)(p_0^h+p_0^{h-1}-p_0^{h-2}-p_0-1)+p_0+1\ (\ast)$$
which is at most $p_0^{hk}+p_0^{hk-1}+p_0^{hk-2}+3p_0^{hk-3}$. This yields $y\leq  3p_0^{hk-h-3}$.

\end{proof}

\begin{theorem}\label{aantalsecanten} A point of a small minimal $k$-blocking set $B$ with exponent $e$ in $\PG(n,p_0^h)$, $p_0:=p^e\geq 7$, $p$ prime, $k>1$, lying on a $(p_0+1)$-secant, lies on at least  $((p_0^{hk}-1)/(p_0^h-1)-3p_0^{hk-h-3})(p_0^{h-1}-4p_0^{h-2})+1$  $(p_0+1)$-secants.
\end{theorem}
\begin{proof} Let $P$ be a point on a $(p_0+1)$-secant $L$. By Lemma \ref{e}, there is an $(n-k)$-space $\Pi$ through $L$ such that $B\cap \Pi=B\cap L$.
Let $\Sigma$ be a small $(n-k+1)$-space. It is clear that the space $\Sigma$ meets $B$ in a small $1$-blocking set $B'$. Every $(n-k)$-space contained in $\Sigma$ meets $B'$ in $1$ mod $p_0$ points. By Theorem \ref{szonyi} (ii), $B'$ is a small minimal $1$-blocking set in $\Sigma$. For every small $(n-k+1)$-space $\Sigma_i$ through $\pi$, $P$ is a point in $\Sigma_i$, lying on a $(p_0+1)$-secant in $\Sigma_i$, and hence, by Lemma \ref{lemma1}, $P$ lies on at least $p_0^{h-1}-4p_0^{h-2}+1$ $(p_0+1)$-secants to $B$ in $\Sigma_i$. From Lemma \ref{large}, we get that the number of small $(n-k+1)$-spaces $\Sigma_i$ through $\Pi$ is at least $(p_0^{hk}-1)/(p_0^h-1)-3p_0^{hk-h-3}$, hence, the number of $(p_0+1)$-secants to $B$ through $P$ is at least $((p_0^{hk}-1)/(p_0^h-1)-3p_0^{hk-h-3})(p_0^{h-1}-4p_0^{h-2})+1$.
\end{proof}

We will now show that Theorem \ref{sziklai} (ii) can be extended to $k$-blocking sets in $\PG(n,q)$. We start with the case $k=1$.
\begin{lemma} \label{pplus1} Let $B$ be a small minimal $1$-blocking set with exponent $e$ in $\PG(n,q)$, $q=p^t$. If for a certain line $L$, $\vert L\cap B \vert=p^e+1$, then $\F_{p^e}$ is a subfield of $\F_q$ and $L\cap B$ is $\F_{p^e}$-linear.
\end{lemma}
\begin{proof} We proceed by induction on $n$. For $n=2$, the statement follows from Theorem \ref{linrechte} (ii), hence, let $n>2$. Let $L$ be a line, meeting $B$ in $p^e+1$ points and let $H$ be a hyperplane through $L$. A plane through $L$ containing a point of $B$, not on $L$, contains at least $p^{2e}$ points of $B$, not on $L$ by Theorem \ref{szonyi} (i). If all $q^{n-2}$ planes through $L$, not in $H$, contain an extra point of $B$, then $\vert B \vert \geq p^{2e}q^{n-2}$, which is larger than $p^h+p^{h-1}+p^{h-2}+3p^{h-3}$, a contradiction by Lemma \ref{groottes}. Let $Q$ be a point on a plane $\pi$ through $L$, not in $H$ such that $\pi$ meets $B$ only in points of $L$. The projection of $B$ onto $H$ is a small minimal $1$-blocking set $B'$ in $H$ (see Theorem \ref{szonyi} (iii)), for which $L$ is a $(p^e+1)$-secant. The intersection $B'\cap L$ is by the induction hypothesis an $\F_{p^e}$-linear set. Since $B\cap L=B'\cap L$, the statement follows.
\end{proof}

Finally, we extend Theorem \ref{sziklai} (ii) to a theorem on $k$-blocking sets in $\PG(n,q)$.
\begin{theorem} \label{pplus1k} Let $B$ be a small minimal $k$-blocking set with exponent $e$ in $\PG(n,q)$, $q=p^t$. If for a certain line $L$, $\vert L\cap B \vert=p^e+1$, $p^e\geq 7$, then $\F_{p^e}$ is a subfield of $\F_q$ and $L\cap B$ is $\F_{p^e}$-linear.
\end{theorem}
\begin{proof} Let $L$ be a $p^e+1$-secant to $B$. By Lemma \ref{large}, there is at least one small $(n-k+1)$-space $\Pi$ through $L$. Since $\Pi\cap B$ is a small $1$-blocking set to $B$, and every $(n-k)$-space, contained in $\Pi$ meets $B$ in $1$ mod $p^e$ points, by Theorem \ref{szonyi} (ii), $B$ is minimal. By Lemma \ref{pplus1}, $L\cap B$ is an $\F_{p^e}$-linear set.
\end{proof}

\section{The proof of Main Theorem 1}

In this section, we will prove Main Theorem 1, that, roughly speaking, states that if we can prove the linearity for $k$-blocking sets in $\PG(n,q)$ for a certain value of $n$, then it is true for all $n$. It is clear from the definition of a $k$-blocking set that we can only consider $k$-blocking sets in $\PG(n,q)$ where $1\leq k\leq n-1$, and whenever we use the notation $k$-blocking set in $\PG(n,q)$, we assume that the above condition is satisfied.

\begin{itemize}
\item[]
From now on, if we want to state that for the pair $(k,n^*)$,  all small minimal $k$-blocking sets in $\PG(n^*,q)$ are linear, we say that the condition ($H_{k,n^*}$) holds.
\end{itemize}
To prove Main Theorem 1, we need to show that if ($H_{k,n^\ast})$ holds, then  $(H_{k,n})$ holds for all $n\geq k+1$.
The following observation shows that we only have to deal with the case $n\geq n^\ast$.
\begin{lemma} \label{simpel}If ($H_{k,n^\ast})$ holds, then $(H_{k,n})$ holds for all $n$ with $k+1\leq n\leq n^\ast$.
\end{lemma}
\begin{proof} A small minimal $k$-blocking set $B$ in $\PG(n,q)$, with $k+1\leq n\leq n^\ast$, can be embedded in $\PG(n^\ast,q)$, in which it clearly is a small minimal $k$-blocking set. Since ($H_{k,n^\ast}$) holds, $B$ is linear, hence, ($H_{k,n}$) holds.
\end{proof}

The main idea for the proof of Main Theorem 1 is to prove that all the $(p_0+1)$-secants through a particular point $P$ of a $k$-blocking set $B$ span a $hk$-dimensional space $\mu$ over $\F_{p_0}$, and to prove that the linear blocking set defined by $\mu$ is exactly the $k$-blocking set $B$. 
\begin{lemma} \label{lemma9} Assume ($H_{k,n-1}$) and $n-1\geq 2k$, and let $B$ denote a small minimal $k$-blocking set with exponent $e$ in $\PG(n,p^t)$, $p$ prime, $p^e\geq 7$, $t\geq 2$.  Let $\Pi$ be a plane in $\PG(n,p^t)$.
\begin{itemize}
\item[(i)] There is a $3$-space $\Sigma$ through $\Pi$ meeting $B$ only in points of $\Pi$ and containing a point $Q$ not lying on a secant line to $B$ {\bf if} $\mathbf{k>2}$. 
\item[(ii)]The intersection $\Pi\cap B$, is a linear set {\bf if} $\mathbf{k>2}$.
\end{itemize}
\end{lemma}
\begin{proof} Let $\Pi$ be a plane of $\PG(n,p^t)$, $p_0:=p^e\geq 7$. By Lemma \ref{lemma6}, there are at least $$s:=(p_0^{h(n+1)}-1)/(p_0^h+1)-(p_0^{2hk-2}+2p_0^{2hk-3})(p_0^h+1)-p_0^{hk}-p_0^{hk-1}-p_0^{hk-2}-3p_0^{hk-3},$$ points $Q\notin \{B\}$ not lying on a secant line to $B$. This means that there are at least $r:=(s-(p_0^{2h}+p_0^h+1))/p_0^{3h}$ $3$-spaces through $\Pi$ that contain a point that does not lie on a secant line to $B$ and is not contained in $B$ nor in $\Pi$.
If all $r$ $3$-spaces contain a point $Q$ of $B$ that is not contained in $\Pi$, then the number of points in $B$ is at least $r$. It is easy to check that this is a contradiction if $n-1\geq 2k$, $p^e\geq 7$, and $k>2$.

Hence, there is a $3$-space $\Sigma$ through $\Pi$ meeting $B$ only in points of $\Pi$ and containing a point $Q$ not lying only on a secant line to $B$. 
The projection of $B$ from $Q$ onto a hyperplane containing $\Pi$ is a small minimal $k$-blocking set $\bar{B}$ in $\PG(n-1,q)$ (see Theorem \ref{szonyi}(iii)), which is, by ($H_{k,n-1}$), a linear set. Now $\Pi\cap \bar{B}=\Pi\cap B$, since the space $\langle Q,\pi\rangle$ meets $B$ only in points of $\Pi$, and hence, the set $\Pi\cap B$ is linear.
\end{proof}

\begin{corollary} \label{rechte} Assume ($H_{k,n-1}$), $k>2$, $(n-1)\geq 2k$ and let $B$ denote a small minimal $k$-blocking set with exponent $e$ in $\PG(n,p^t)$, $p$ prime, $p^e\geq 7$, $t\geq 2$.  The intersection of a line with $B$ is an $\F_{p^e}$-linear set.
\end{corollary}


\begin{remark}\label{rem2} The linear set $\B(\mu)$  does not determine the subspace $\mu$ in a unique way; by Remark \ref{trans}, we can choose $\mu$ through a fixed point $S(P)$, with $P\in \B(\mu)$. Note that there may exist different spaces $\mu$ and $\mu'$, {\em through the same point of $\PG(h(n+1)-1,p)$}, such that $\B(\mu)=\B(\mu')$. 
If $\mu$ is a line, however, if we fix a point $x$ of an element of $\B(\mu)$, then there is a unique line $\mu'$ through $x$ such that $\B(\mu)=\B(\mu')$ since, in this case, $\mu'$ is the unique transversal line through $x$ to the regulus $\B(\mu)$. This observation is crucial for the proof of the following lemma.
\end{remark}

\begin{lemma}\label{span} Assume ($H_{k,n-1}$), $n-1\geq 2k$, and let $B$ be a small minimal $k$-blocking set with exponent $e$ in $\PG(n,p^t)$, $p$ prime, $p_0:=p^e\geq 7$. Denote the $(p_0+1)$-secants through a point $P$ of $B$ that lies on at least one $(p_0+1)$-secant, by $L_1,\ldots,L_{s}$. Let $x$ be a point of $\S(P)$ and let $\ell_i$ be the line through $x$ such that $\B(\ell_i)=L_i\cap B$. The following statements hold:
\begin{itemize}
\item[(i)] The space $\langle \ell_1,\ldots,\ell_s\rangle$ has dimension $hk$.
\item[(ii)] $\B(\langle \ell_i,\ell_j\rangle)\subseteq B$ for $1\leq i\neq j\leq s$.
\end{itemize}
\end{lemma}
\begin{proof} 
(i) Let $P$ be a point of $B$ lying on a $(p_0+1)$-secant, and let $H$ be a hyperplane through $P$. By Lemma 6, there is a point $Q$, not in $B$ and not in $H$, not lying on a secant line to $B$.
The projection of $B$ from $Q$ onto $H$ is a small minimal $k$-blocking set $\bar{B}$ in $H\cong \PG(n-1,q)$ (Theorem \ref{szonyi} (iii)). By $(H_{k,n-1})$, $\bar{B}$ is a linear set. Every line meets $B$ in $1$ mod $p_0$ or $0$ points, which implies that every line in $H$ meets $\bar{B}$ in $1$ mod $p_0$ or $0$ points, hence, $\bar{B}$ is $\F_{p_0}$-linear.
Take a fixed point $x$ in $\S(P)$. Since $\bar{B}$ is an $\F_{p_0}$-linear set, there is an $hk$-dimensional space $\mu$ in $\PG(h(n+1)-1,p_0)$, through $x$, such that $\B(\mu)=\bar{B}$.

From Lemma \ref{lemma1}, we get that the number of $(p_0+1)$-secants through $P$ to $B$ is at least $z:=((p_0^{hk}-1)/(p_0^h-1)-3p_0^{hk-h-3})(p_0^{h-1}-4p_0^{h-2})+1$, denote them by $L_1,\ldots,L_s$ and let $\ell_1,\ldots,\ell_s$ be the lines through $x$ such that $\B(\ell_i)=B\cap L_i$. These lines exist by Theorem \ref{pplus1k}. Note that, by Remark \ref{rem2}, $\B(\ell_i)$ determines the line $\ell_i$ through $x$ in a unique way, and that $\ell_i\neq \ell_j$ for all $i\neq j$.

We will prove that the projection of $\ell_i$ from $\S(Q)$ onto $\langle \S(H)\rangle$ in $\PG(h(n+1)-1,p_0)$ is contained in $\mu$. Since $L_1$ is projected onto a $(p_0+1)$-secant $M$ to $\bar{B}$ through $P$, there is a line $m$ through $x$ in $\PG(h(n+1)-1,p_0)$ such that $\B(m)=M\cap \bar{B}$. Now $\bar{B}=\B(\mu)$, and $\vert\bar{B}\cap M\vert=p_0+1$, hence, there is a line $m'$ through $x$ in $\mu$ such that $\B(m')=\bar{B}\cap M$. Since $m$ is the unique transversal line through $x$ to $M\cap \bar{B}$ (see Remark \ref{rem2}), $m=m'$, and $m$ is contained in $\mu$.

This implies that the space $W:=\langle \ell_1,\ldots,\ell_s\rangle$ is contained in $\langle \S(Q),\mu\rangle$, hence, $W$ has dimension at most $hk+h$. 
Suppose that $W$ has dimension at least $hk+1$, then it intersects the $(h-1)$-dimensional space $\S(Q)$ in at least a point. But this holds for all $\S(Q)$ corresponding to points, not in $B$, such that $Q$ does not lie on a secant line to $B$. This number is at least $$(p_0^{h(n+1)}-1)/(p_0^h+1)-(p_0^{2hk-2}+2p_0^{2hk-3})(p_0^h+1)-p_0^{hk}-p_0^{hk-1}-p_0^{hk-2}-3p_0^{hk-3}$$ by Lemma \ref{lemma6}, which is larger than the number of points in $W$, since $W$ is at most $(hk+h)$-dimensional, a contradiction.

From Theorem \ref{aantalsecanten}, we get that $W$ contains at least $$(((p_0^{hk}-1)/(p_0^h-1)-3p_0^{hk-h-3})(p_0^{h-1}-4p_0^{h-2})+1)p_0+1$$ points, which is larger than $(p_0^{hk}-1)/(p_0-1)$ if $p_0\geq 7$, hence, $W$ is at least $hk$-dimensional. Since we have already shown that $W$ is at most $hk$-dimensional, the statement follows.\\

(ii) W.l.o.g. we choose $i=1,j=2$. Let $m$ be a line in $\langle \ell_1,\ell_2\rangle$, not through $\ell_1\cap \ell_2$. 
Let $M$ be the line of $\PG(n,q^t)$ containing $\B(m)$ and let $H$ be a hyperplane of $\PG(n,q^t)$ containing the plane $\langle L_1,L_2\rangle$. 
We claim that there exists a point $Q$, not in $H$, such that the planes $\langle Q,L_1\rangle, \langle Q,L_2\rangle$ and $\langle Q,M\rangle$ only contain points of $B$ that are in $H$.

If $k>2$, this follows from Lemma \ref{lemma9}(i). Now assume that $1\leq k\leq 2$. There are $q^{n-2}$ planes through $M$, not in in $H$. Since $M$ is at least a $(p_0+1)$-secant (Theorem \ref{szonyi} (i)), it holds that if a plane $\Pi$ through $M$ contains a point of $B$, that is not contained in $M$, then, $\Pi$ contains at least $p_0^2$ points of $B$, not in $M$ (again by Theorem \ref{szonyi}(i)). Since $\vert B\vert \leq q^k+q^{k-1}+q^{k-2}+3q^{k-3}$ (Lemma \ref{groottes}), and $n-1\geq 2k$, there is at least one plane $\Pi$ through $M$, not contained in $H$ that contains only points of $B$ that are contained in $M$. Now, there is one of the $q^2$ points in $\Pi$, say $Q$, that is not contained in $M$ for which  the planes $\langle Q,L_i\rangle$, $i=1,2$ only contain points of $B$ on the line $L_i$, $i=1,2$, since otherwise, the number of points in $B$ would be at least $p_0^2q^2$, a contradiction since $k\leq 2$ and $\vert B\vert \leq q^k+q^{k-1}+q^{k-2}+3q^{k-3}$ by Lemma \ref{groottes}. This proves our claim.

The projection of $B$ from $Q$ onto $H$ is a small minimal $k$-blocking set $\bar{B}$ in $\PG(n,q)$ (Theorem \ref{szonyi} (iii)). By $(H_{k,n-1})$, $\bar{B}$ is a linear set, hence, it meets $\langle L_1,L_2\rangle$ in a linear set. This means that there is a space $\pi$ through $x$ such that $\langle L_1,L_2\rangle\cap B=\mathcal{B}(\pi)$. Note that, since $\langle Q,L_1\rangle$ and $\langle Q,L_2\rangle$ only contain points of $B$ that are contained in $H$, the lines $L_1$ and $L_2$ are $(p_0+1)$-secants to $\bar{B}$. 

Hence, the space $\pi$ contains $\ell_i$  since $\B(\pi)\cap L_i=\B(\ell_i)$ and $\ell_i$ is the unique transversal line to the regulus $B\cap L_i$, $i=1,2$. Hence, $\B(\langle \ell_1,\ell_2\rangle)\subset \bar{B}$, so $\B(m)\subset \bar{B}$. The plane $\langle Q,M\rangle$ only contains points of $B$ that are on $M$, so $M\cap B=M\cap \bar{B}$, hence, $\B(m)\subset B$. Since every point of $\langle \ell_1,\ell_2\rangle$, not on $\ell_1,\ell_2$, lies on a line $m$ meeting $\ell_1$ and $\ell_2$ in different points, $\B(\langle \ell_1,\ell_2\rangle)\subseteq B$.
\end{proof}

{\bf Proof of Main Theorem 1.} 

 Let $B$ be a small minimal $k$-blocking set with exponent $e$ in $\PG(n,p^t)$, $p$ prime, $p_0=p^e\geq 7$ and assume that $(H_{k,n-1})$ holds with $n-1\geq 2k$. 
Let $P$ be a point of $B$, lying on a $(p_0+1)$-secant. By Theorem \ref{aantalsecanten}, there are at least $((p_0^{hk}-1)/(p_0^h-1)-3p_0^{hk-h-3})(p_0^{h-1}-4p_0^{h-2})+1$  $(p_0+1)$-secants $L_1\ldots,L_s$ through $P$, and by Lemma \ref{span}, the corresponding lines $\ell_1,\ldots,\ell_s$ in $\PG(h(n+1)-1,p_0)$, with $\B(\ell_i)=B\cap L_i$, $\ell_i$ through a fixed point $x$ of $\S(P)$, span an $hk$-dimensional space $W$. 
Suppose that $\B(W)\not\subseteq B$, and let $w$ be a point of $W$ for which $\B(w)\notin B$. Since the number of points lying on one of the lines of the set $\{\ell_1,\ldots,\ell_s\}$, is at least $(((p_0^{hk}-1)/(p_0^h-1)-3p_0^{hk-h-3})(p_0^{h-1}-4p_0^{h-2})+1)p_0+1$, at least one of the $(p_0^{hk}-1)/(p_0-1)$ lines through $w$, say $m$, contains two points lying on one of the lines of the set $\{\ell_1,\ldots,\ell_s\}$. By Lemma \ref{span} (b), $\B(m)$ is contained in $B$, a contradiction since $\B(w)\in \B(m)$, and $\B(w)\notin B$. 

Hence, $\B(W)\subseteq B$, and since $\B(W)$ is a small minimal linear $k$-blocking set $\PG(n,p^t)$, contained in the minimal $k$-blocking set $B$, $B$ equals the linear set $\B(W)$. Hence, we have shown that if $(H_{k,n-1})$ holds, with $n-1\geq 2k$, then $(H_{k,n})$ holds, and repeating this argument shows that if $(H_{k,n^*})$ holds for some $n^*$, then  $(H_{k,n})$ holds for all $n\geq n^*$. Since Lemma \ref{simpel} shows the desired property for all $n$ with $k+1\leq n \leq n^*$, the statement follows. \qed

\section{The proof of Main Theorem 2}
In this section, we will prove Main Theorem 2, stating that, if all small minimal $1$-blocking sets in $\PG(n,p_0^h)$ are linear, then all small minimal $k$-blocking sets in $\PG(n,p_0^h)$, are linear, provided a condition on $p_0$ and $h$ holds.

We proved in Lemma \ref{groottes} that a subspace meets the small minimal $k$-blocking set $B$ in either in a `small' number, or in a `large' number of points. To simplify the terminology, we call a $(n-k+s)$-space $\Pi$, $s\leq k$,  for which $\vert B\cap \Pi \vert< p_0^{hs}+p_0^{hs-1}+p_0^{hs-2}+3p_0^{hs-3}$ points, a {\em small} $(n-k+s)$-space. An $(n-k+s)$-space which is not small is called {\em large}.

%




%

%

\begin{lemma}\label{hypervlakken} Let $\Pi$ be an $(n-k)$-space of $\PG(n,p_0^h)$ and let $B$ be a small minimal $k$-blocking set with exponent $e$ in $\PG(n,p^t)$, $p$ prime, $p_0:=p^e\geq 7$, $k>1$.
\begin{enumerate}
\item[(i)] If $B\cap \Pi$ is a point, then there are at most $p_0^{hk-h-2}+4p_0^{hk-h-3}-1$ large $(n-k+1)$-spaces through $\Pi$.
\item[(ii)] 
If $\Pi$ intersects $B$ in $p_0+1$ points, then there are at most $3p_0^{hk-h-3}$ large $(n-k+1)$-spaces through $\Pi$.
\end{enumerate}
\end{lemma}
\begin{proof} (i) A small $(n-k+1)$-space through $\Pi$ meets $B$ in at least $p_0^h+1$ points. Suppose there are $y$ large $(n-k+1)$-spaces through $\Pi$. Then the number of points in $B$ is at least
$$y(p_0^{h+1}-p_0^{h-1}-p_0^{h-2}-3p_0^{h-3}-1)+
((p_0^{hk}-1)/(p_0^h-1)-y)p_0^h+1$$
which is at most $p_0^{hk}+p_0^{hk-1}+p_0^{hk-2}+3p_0^{hk-3}$. This yields $y\leq  p_0^{hk-h-2}+4p_0^{hk-h-3}-1$.

(ii) Suppose there are $y$ large $(n-k+1)$-spaces through $\Pi$. A small $(n-k+1)$-space through $\Pi$ meets $B$ in a linear $1$-blocking set, which is in this case, non-trivial and hence, by Theorem \ref{grootte}, has at least $p_0^h+p_0^{h-1}-p_0^{h-2}$ points.

Then the number of points in $B$ is at least
$$ y(p_0^{h+1}-p_0^{h-1}-p_0^{h-2}-3p_0^{h-3}-p_0-1)+$$
$$((p_0^{hk}-1)/(p_0^h-1)-y)(p_0^h+p_0^{h-1}-p_0^{h-2}-p_0-1)+p_0+1\ (\ast)$$
which is at most $p_0^{hk}+p_0^{hk-1}+p_0^{hk-2}+3p_0^{hk-3}$. This yields $y\leq  3p_0^{hk-h-3}$.
\end{proof}

\begin{lemma}\label{situatie} If $B$ is a non-trivial small minimal $k$-blocking set with exponent $e$ in $\PG(n,p^t)$, $p$ prime, $p_0:=p^e\geq 7$, $k>1$, then there exist a point $P\in B$, a tangent $(n-k)$-space $\Pi$ at the point $P$  and  small $(n-k+1)$-spaces $H_i$, through $\Pi$, such that there is a $(p_0+1)$-secant through $P$ in $H_i$, $i=1,\ldots,p_0^{hk-h}-5p_0^{hk-h-1}$.
\end{lemma}
\begin{proof} Let $L$ be a $(p_0+1)$-secant to $B$ and let $P$ be a point of $B\cap L$. Lemma \ref{e} shows that there is an $(n-k)$-space $\Pi_L$ such that $B\cap \Pi_L=B\cap L$.
By Theorem \ref{aantalsecanten}, $P$ lies on $((p_0^{hk}-1)/(p_0^h-1)-3p_0^{hk-h-3})(p_0^{h-1}-4p_0^{h-2})+1$ other $(p_0+1)$-secants. 
By Lemma \ref{hypervlakken} (ii), there are at least $(p_0^{hk}-1)/(p_0^h-1)-3p_0^{hk-h-3}$ small hyperplanes through $\Pi_L$, which each contain at least $p_0^h+p_0^{h-1}-p_0^{h-2}-p_0-1$ points of $B$ not on $L$. Since $\vert B \vert <p_0^{hk}+p_0^{hk-1}+p_0^{hk-2}+3p_0^{hk-3}$ (see Lemma \ref{grootte}), there are less than $2p_0^{hk-1}$ points of $B$ left in large $(n-k+1)$-spaces through $\Pi_L$. Hence, $P$ lies on less than $2p_0^{hk-h-1}$ lines that are completely contained in $B$.

Since $B$ is minimal, $P$ lies on a tangent $(n-k)$-space $\Pi$ to $B$. There are at most $p_0^{hk-h-2}+4p_0^{hk-h-3}-1$ large $(n-k+1)$-spaces through $\Pi$ (Lemma \ref{hypervlakken} (i)). Moreover, since at least $\frac{p_0^{hk}-1}{p_0^h-1}-(p_0^{hk-h-2}+4p_0^{hk-h-3}-1)-(2p_0^{hk-h-1})$ $(n-k+1)$-spaces through $\Pi$ contain at least $p_0^h+p_0^{h-1}-p_0^{h-2}$ points of $B$, and at most $2p_0^{hk-h-1}$ of the small $(n-k+1)$-spaces through $\Pi$ contain exactly $p_0^h+1$ points of $B$, there are at most $p_0^{hk-2}$ points of $B$ contained in large $(n-k+1)$-spaces through $\Pi$. Hence, $P$ lies on at most $p_0^{hk-3}$ $(p_0+1)$-secants of the large $(n-k+1)$-spaces through $\Pi$. This implies that there are at least 
$(((p_0^{hk}-1)/(p_0^h-1)-3p_0^{hk-h-3})(p_0^{h-1}-4p_0^{h-2})+1)-p_0^{hk-3}$
 $(p_0+1)$-secants through $P$ left in small $(n-k+1)$-spaces through $\Pi$. Since in a small $(n-k+1)$-space through $\Pi$, there can lie at most $(p_0^{h}-1)/(p_0-1)$ $(p_0+1)$-secants through $P$, this implies that there are at least $p_0^{hk-h}-5p_0^{hk-h-1}$ $(n-k+1)$-spaces $H_i$ through $\Pi$ such that $P$ lies on a $(p_0+1)$-secant in $H_i$.
\end{proof}

We continue with the following hypothesis:
\begin{itemize}
\item[(H)] A small minimal $j$-blocking set in $\PG(n,q)$, $1\leq j<k$ is linear. 
\end{itemize}

\begin{lemma}  \label{handig} Let $B$ be a non-trivial small minimal $k$-blocking set with exponent $e$ in $\PG(n,p^t)$, $p$ prime, $p_0:=p^e\geq 7$, $k>1$. If we assume (H), then the following statements hold.
\begin{itemize}
\item[(i)]  A small $(n-k+s)$-dimensional space $\Pi$ of $\PG(n,p^t)$, $s<k$, 
intersects $B$ in a linear set and $|\Pi \cap B|\leq (p_0^{hs+1}-1)/(p_0-1)$.
\item[(ii)] Let $L$ be a $(p_0+1)$-secant to $B$ and let $S$ be a point of $B$, not on $L$. There exists a small $(n-2)$-space through $L$, skew to $S$.
\item[(iii)] A line intersects $B$ in a linear set.
\item[(iv)] Let $\Pi$ be a small $(n-2)$-space containing a $(p_0+1)$-secant to $B$. Then the number of large $(n-1)$-spaces through $\Pi$ is at most $4p_0^{h-3}$.
\end{itemize}
\end{lemma}
\begin{proof} 
(i) It is clear that an $(n-k+s)$-space $\Pi$ meets $B$ in a small $s$-blocking set $B'$. Every $(n-k)$-space contained in $\Pi$ meets $B'$ in $1$ mod $p_0$ points, hence, by Theorem \ref{szonyi} (ii), $B'$ is a small minimal $s$-blocking set in $\PG(n-k+s,p_0^h)$, which is, by the hypothesis (H), $\F_{p_0}$-linear. It follows that $|B'|\leq (p_0^{hs+1}-1)/(p_0-1)$.

(ii) Lemma \ref{e} shows that there is an $(n-k)$-space $\Pi_{n-k}$ through $L$, such that $B\cap L=B\cap \Pi_{n-k}$. 
By Lemma \ref{groottes}, an $(n-k+1)$-space through $\Pi_{n-k}$ contains at most $(p_0^{h+1}-1)/(p_0-1)$ or at least $p_0^{h+1}-p_0^{h-1}-p_0^{h-2}-3p_0^{h-3}$ points of $B$. If all $(n-k+1)$-spaces through $\Pi_{n-k}$ (except possibly $\langle \Pi_{n-k},S\rangle$) would be large, the number of points in $B$ would be at least $((p_0^{hk}-1)/(p_0^h-1)-1)(p_0^{h+1}-p_0^{h-1}-p_0^{h-2}-3p_0^{h-3}-p_0^h)$, which is larger than  $p_0^{hk}+p_0^{hk-1}+p_0^{hk-2}+3p_0^{hk-3}$, a contradiction. Hence, there is a small $(n-k+1)$-space through $\Pi_{n-k}$.

Suppose, by induction, that there exists a small $(n-k+s)$-space $\Pi_{n-k+s}$ through $L$, skew to $S$ and suppose all $(p_0^{h(k-s)}-1)/(p_0^h-1)-1$ $(n-k+s)$-spaces through $\Pi_{n-k+s-1}$, different from $\langle \Pi_{n-k+s},S\rangle$ are large. Then the number of points in $B$ is larger than $p_0^{hk}+p_0^{hk-1}+p_0^{hk-2}+3p_0^{hk-3}$ if $s\leq k-2$, a contradiction. We conclude that there exists a small $(n-2)$-space through $L$, skew to $S$.

(iii) Let $L$ be a line, with $0<|L\cap B|<p^t+1$, otherwise the statement trivially holds. The previous part of this lemma shows that $L$ is contained in a small $(n-k+1)$-space, which has, by the first part of this lemma, a linear intersection with $B$. Hence, $B\cap L$ is a linear set.

(iv) A small $(n-1)$-space through $\Pi$ meets $B$ in at least $p_0^{hk-h}+p^{hk-h-1}-p^{hk-h-2}$ points (see Corollary \ref{grootte}) and a small $(n-2)$-space contains at most $(p_0^{hk-2h+1}-1)/(p_0-1)$ points by the first part of this lemma. By Lemma \ref{groottes}, a large $(n-1)$-space through $\Pi$ contains at least $p^{hk-h+1}-p^{hk-h-1}-p^{hk-h-2}-3p^{hk-h-3}$ points of $B$. Suppose there are $y$ large $(n-1)$-spaces through $\Pi$. Then the number of points in $B$ is at least
$$y(p_0^{hk-h+1}-p_0^{hk-h-1}-p_0^{hk-h-2}-3p_0^{hk-h-3}-(p_0^{hk-2h+1}-1)/(p_0-1))+$$
$$(p_0^h+1-y)(p_0^{hk-h}+p^{hk-h-1}-p^{hk-h-2}-(p_0^{hk-h+1}-1)/(p_0-1))+(p_0^{hk-2h+1}-1)/(p_0-1)$$
which is at most $p_0^{hk}+p_0^{hk-1}+p_0^{hk-2}+3p_0^{hk-3}$. This yields $y\leq  4p_0^{h-3}$.
\end{proof}

\begin{lemma}\label{ess} Assume (H). Let $B$ be a non-trivial small minimal $k$-blocking set with exponent $e$ in $\PG(n,p^t)$, $p$ prime, $p_0:=p^e\geq 7$ and let $P$ be a point of $B$, and let $\Pi$ be a tangent $(n-k)$-space to $B$ through $P$. Let $H_1$ and $H_2$ be two $(n-k+1)$-spaces through $\Pi$ for which $B\cap H_i=\mathcal{B}(\pi_i)$, for some $h$-space $\pi_i$ through a point $x\in \mathcal{S}(P)$, such that $P$ lies on a $(p_0+1)$-secant in $H_i$, $i=1,2$. Then $\mathcal{B}(\langle \pi_1,\pi_2 \rangle)\subset B$.
\end{lemma}
\begin{proof}
Let $L$ be a $(p_0+1)$-secant through $P$ in $H_1$ and let $\ell$ be the line in $\pi$ through $x$ such that $\langle\B(\ell)\rangle=L$. Let $s$ be a point of $\pi_2$. By Lemma \ref{handig} (ii), there is a small $(n-2)$-space $\Pi_{n-2}$ through $L$, skew to $\B(s)$. There are at least $p_0^{h-1}-4p_0^{h-2}$ $(p_0+1)$-secants through $P$, of which at least $p_0^{h-1}-4p_0^{h-2}-(p_0^{h-1}-1)/(p_0-1)$ span an $(n-1)$-space together with $\Pi_{n-2}$. By Lemma \ref{handig} (iv), there are at most $4p_0^{h-3}$ large spaces through $\Pi_{n-2}$, so at least $p_0^{h-1}-4p_0^{h-2}-(p_0^{h-1}-1)/(p_0-1)-4p_0^{h-3}$ of the $(p_0+1)$-secants through $P$ have a transversal line $\ell_k$, for which $\B(\langle\ell,\ell_k\rangle)\subset B$. This gives in total at least $p_0^{h+1}-6p_0^{h}$ points $Q$ in $\langle \ell,\pi_2\rangle$ for which $\B(Q)\subset B$, denote this pointset by $G$. This means that every point $t$ of $\langle \ell,\pi_2\rangle$ lies on a line $m$ with at least $p_0-5$ points of $G$. Since $\langle \B(m)\rangle$ either is contained in $B$, or it meets $B$ in a linear set of rank at most $h$ (see Lemma \ref{handig} (iii)), and $p_0-5>h$, again by Theorem \ref{subline}, $\B(m)\subset B$ by Theorem \ref{subline}, and hence, $\B(t)\subset B$. 

Hence, for all $(p_0+1)$-secants $\B(\ell)$, with $\ell$ through $x$, in $H_1$, $\B(\langle \ell,\pi_2\rangle)\subset B$. This shows that there are at least $(p_0^{h-1}-4p_0^{h-2})p_0^{h+1}+(p_0^{h+1}-1)/(p_0-1)$ points $Q$ in the $2h$-space $\langle \pi_1,\pi_2\rangle$ such that $\B(Q)\subset B$. Every point $t$ of 
$\langle \pi_1,\pi_2\rangle$ lies on a line $m$ with at least $p_0-5$ points of $G$. Again, since $p_0-5>h$, by Theorem \ref{subline}, $\B(m)\subset B$ and hence, $\B(t)\subset B$. It follows that $\B(\langle \pi_1,\pi_2\rangle)\subseteq B$.

\end{proof}

{\bf Proof of Main Theorem 2.} Let $B$ be a non-trivial small minimal $k$-blocking set with exponent $e$ in $\PG(n,p^t)$, $p$ prime, $p_0:=p^e\geq 7$.
We will show that, assuming that all small minimal $1$-blocking sets with exponent $e$ in $\PG(n,p^t)$, $p$ prime, $p_0:=p^e\geq 7$, are $\mathbb{F}_{p_0}$-linear, $B$ is $\mathbb{F}_{p_0}$-linear. By induction, we may assume (H) holds. If $B$ is a $k$-space, then $B$ is $\mathbb{F}_{p_0}$-linear. If $B$ is a non-trivial small minimal $k$-blocking set, Lemma \ref{situatie} shows that there exists a point $P$ of $B$, a tangent $(n-k)$-space $\Pi$ at the point $P$ and at least $p_0^{hk-h}-5p_0^{hk-h-1}$ $(n-k+1)$-spaces $H_i$  through $\Pi$ for which $B\cap H_i$ is small and linear, where $P$ lies on at least one $(p_0+1)$-secant of $B\cap H_i$, $i=1,\ldots,s$, $s\geq p_0^{hk-h}-5p_0^{hk-h-1}$. Let $B\cap H_i=\mathcal{B}(\pi_i), i=1,\ldots, s$, with $\pi_i$ an $h$-dimensional space in $\PG(h(n+1)-1,p_0)$, where $x\in \pi_i$, with $x\in \S(P)$.

Lemma \ref{ess} shows that $\mathcal{B}(\langle \pi_i,\pi_j \rangle)\subseteq B$, $0\leq i\neq j\leq s$. 

If $k=2$, the set $\mathcal{B}(\langle \pi_1,\pi_2 \rangle)$ corresponds to a linear $2$-blocking set $B'$ in $\PG(n,p_0^h)$. Since $B$ is minimal, $B=B'$, and the Theorem is proven.

Let $k>2$. Denote the $(n-k+1)$-spaces through $\Pi$, different from $H_i$, by $K_j, j=1,\ldots, z$. It follows from Lemma \ref{situatie} that $z\leq 5p_0^{hk-h-1}+(p_0^{hk-h}-1)/(p_0-1)\leq 6p_0^{hk-h-1}$.
There are at least $(p_0^{hk-h}-5p_0^{hk-h-1}-1)/p_0^h$ different $(n-k+2)$-spaces $\langle H_1, H_j\rangle$, $1<j\leq s$. If all $(n-k+2)$-spaces $\langle H_1,H_j\rangle$, contain at least $10p_0^{h-1}$ of the spaces $K_i$, then $z\geq 10p_0^{h-1}(p_0^{hk-h}-5p_0^{hk-h-1}-1)/p_0^h>6p_0^{hk-h-1}$, a contradiction if $p_0> h+10$. Let $\langle H_1,H_2\rangle$ be an $(n-k+2)$-spaces containing less than $10p_0^{h-1}$ spaces $K_i$. 
 
Suppose by induction that for any $1<i<k$, there is an $(n-k+i)$-space $\langle H_1,H_2,\ldots, H_i\rangle$  containing at most $10p_0^{hi-h-1}$ of the spaces $K_i$ such that $\mathcal{B}(\langle \pi_1,\ldots,\pi_i\rangle)\subseteq B$. 

There are at least $$\frac{p_0^{hk-h}-6p_0^{hk-h-1}-(p_0^{hi}-1)/(p_0^h-1)}{p_0^{h}}$$ different $(n-k+i+1)$-spaces $\langle H_1,H_2,\ldots, H_i,H_r\rangle$, $H_r\not\subseteq \langle H_1,H_2,\ldots, H_i\rangle$. If all of these contain at least $10p_0^{hi-1}$ of the spaces $K_i$, then $z\geq 6p_0^{hk-h-1}$,
 a contradiction. Let $\langle H_1,\ldots, H_{i+1}\rangle$ be an $(n-k+i+1)$-space containing less than $10p_0^{hi-1}$ spaces $K_i$. We still need to prove that $\mathcal{B}(\langle \pi_1,\ldots,\pi_{i+1}\rangle)\subseteq B$. Since $\mathcal{B}(\langle \pi_{i+1}, \pi \rangle)\subseteq B$, with $\pi$ an $h$-space in $ \langle \pi_1,\ldots,\pi_i\rangle$ for which $\mathcal{B}(\pi)$ is not contained in one of the spaces $K_i$, there are at most $10p_0^{hi-h-1}$ $2h$-dimensional spaces $\langle \pi_{i+1},\mu\rangle$ for which $\mathcal{B}(\langle\pi_{i+1},\mu\rangle)$ is not necessarily contained in $B$, giving rise to at most $v:=10p_0^{hi-h-1}(p_0^{2h+1}-1)/(p_0-1)$ points $t$ for which $\mathcal{B}(t)$ is not necessarily contained in $B$.  Let $u$ be a point of such a space $\langle \pi_{i+1},\mu \rangle$, and suppose that $\B(u)\notin B$. If each of the $(p_0^{hi+h}-1)/(p_0-1)$ lines through $u$ in $\langle \pi_1,\ldots, \pi_{i+1}\rangle$ contains at least $10$ of the points $t$ for which $\mathcal{B}(t)$ is not in $B$, then there are more than $v$ such points $t$, a contradiction. Hence, there is a line $n$ through $u$ for which for at least $p_0-10$ points $v\in n$, $\mathcal{B}(v)\in B$. Every line $L$ meets $B$ in a linear set (see Lemma \ref{handig} (iii)), and if this linear set has rank at least $h+1$, then $L$ is completely contained in $B$. This implies that $\langle \B(n)\rangle\cap B$ has rank at most $h$, and that the subline $\B(n)$ contains at least $p_0-10$ points of the linear set $\langle \B(n)\rangle\cap B$. Since $p_0-10>h$, by Theorem \ref{subline}, $\B(n)$ is contained in $\langle \B(n)\rangle\cap B$, so $\B(u)\subset B$, a contradiction.

This implies that $\mathcal{B}(\langle \pi_1,\ldots,\pi_{i+1}\rangle)\subseteq B$.
 
Since $\mathcal{B}(\langle \pi_1,\ldots,\pi_k\rangle)\subseteq B$, and $\mathcal{B}(\langle \pi_1,\ldots, \pi_k \rangle)$ corresponds to a linear $k$-blocking set $B'$ in $\PG(n,p_0^h)$ contained in the minimal $k$-blocking set $B$, $B=B'$ and hence, $B$ is $\F_{p_0}$-linear.
\qed
\\

{\bf Acknowledgment:} This research was done while the author was visiting
 the discrete algebra and geometry group (DAM) at Eindhoven University of Technology, the Netherlands. The author thanks A. Blokhuis and all other members of this group for their hospitality during her stay.

\end{document}